\documentclass[11pt]{amsart}
\usepackage{amsopn}
\usepackage{amssymb, amscd}

\newcommand{\nc}{\newcommand}

\nc{\fg}{\mathfrak{f} } \nc{\vg}{\mathfrak{v} } \nc{\wg}{\mathfrak{w} }
\nc{\zg}{\mathfrak{z} } \nc{\ngo}{\mathfrak{n} } \nc{\kg}{\mathfrak{k} }
\nc{\mg}{\mathfrak{m} } \nc{\bg}{\mathfrak{b} } \nc{\ggo}{\mathfrak{g} }
\nc{\ggob}{\overline{\mathfrak{g}} } \nc{\sog}{\mathfrak{so} }
\nc{\sug}{\mathfrak{su} } \nc{\spg}{\mathfrak{sp} } \nc{\slg}{\mathfrak{sl} }
\nc{\glg}{\mathfrak{gl} } \nc{\cg}{\mathfrak{c} } \nc{\rg}{\mathfrak{r} }
\nc{\hg}{\mathfrak{h} } \nc{\tg}{\mathfrak{t} } \nc{\ug}{\mathfrak{u} }
\nc{\dg}{\mathfrak{d} } \nc{\ag}{\mathfrak{a} } \nc{\pg}{\mathfrak{p} }
\nc{\sg}{\mathfrak{s} } \nc{\affg}{\mathfrak{aff} }

\nc{\pca}{\mathcal{P}} \nc{\nca}{\mathcal{N}} \nc{\lca}{\mathcal{L}}
\nc{\oca}{\mathcal{O}} \nc{\mca}{\mathcal{M}} \nc{\tca}{\mathcal{T}}
\nc{\aca}{\mathcal{A}} \nc{\cca}{\mathcal{C}} \nc{\gca}{\mathcal{G}}
\nc{\sca}{\mathcal{S}} \nc{\hca}{\mathcal{H}} \nc{\bca}{\mathcal{B}}
\nc{\dca}{\mathcal{D}} \nc{\val}{\operatorname{val}}

\nc{\vp}{\varphi} \nc{\ddt}{\frac{d}{dt}} \nc{\dds}{\frac{d}{ds}}
\nc{\dpar}{\frac{\partial}{\partial t}} \nc{\im}{\mathtt{i}}

\nc{\SO}{\mathrm{SO}} \nc{\Spe}{\mathrm{Sp}} \nc{\Sl}{\mathrm{SL}}
\nc{\SU}{\mathrm{SU}} \nc{\Or}{\mathrm{O}} \nc{\U}{\mathrm{U}} \nc{\Gl}{\mathrm{GL}}
\nc{\Se}{\mathrm{S}} \nc{\Cl}{\mathrm{Cl}} \nc{\Spein}{\mathrm{Spin}}
\nc{\Pin}{\mathrm{Pin}} \nc{\G}{\mathrm{GL}_n(\RR)} \nc{\g}{\mathfrak{gl}_n(\RR)}

\nc{\RR}{{\Bbb R}} \nc{\HH}{{\Bbb H}} \nc{\CC}{{\Bbb C}} \nc{\ZZ}{{\Bbb Z}}
\nc{\FF}{{\Bbb F}} \nc{\NN}{{\Bbb N}} \nc{\QQ}{{\Bbb Q}} \nc{\PP}{{\Bbb P}}

\nc{\vs}{\vspace{.2cm}} \nc{\vsp}{\vspace{1cm}} \nc{\ip}{\langle\cdot,\cdot\rangle}
\nc{\ipp}{(\cdot,\cdot)} \nc{\la}{\langle} \nc{\ra}{\rangle} \nc{\unm}{\tfrac{1}{2}}
\nc{\unc}{\tfrac{1}{4}} \nc{\und}{\tfrac{1}{16}} \nc{\no}{\vs\noindent}
\nc{\lam}{\Lambda^2(\RR^n)^*\otimes\RR^n} \nc{\tangz}{{\rm T}^{\rm Zar}}
\nc{\nor}{{\sf n}}  \nc{\mum}{/\!\!/} \nc{\kir}{/\!\!/\!\!/}
\nc{\Ri}{\tfrac{4\Ric_{\mu}}{||\mu||^2}} \nc{\ds}{\displaystyle}
\nc{\ben}{\begin{enumerate}} \nc{\een}{\end{enumerate}} \nc{\f}{\frac}
\nc{\lb}{[\cdot,\cdot]} \nc{\isn}{\tfrac{1}{||v||^2}}
\nc{\gkp}{(\ggo=\kg\oplus\pg,\ip)} \nc{\ukh}{(\ug=\kg\oplus\hg,\ip)}
\nc{\tgkp}{(\tilde{\ggo}=\kg\oplus\pg,\ip)}
\nc{\wt}{\widetilde} \nc{\mm}{M}

\nc{\Hess}{\operatorname{Hess}} \nc{\ad}{\operatorname{ad}}
\nc{\Ad}{\operatorname{Ad}} \nc{\rank}{\operatorname{rank}}
\nc{\Irr}{\operatorname{Irr}} \nc{\End}{\operatorname{End}}
\nc{\Aut}{\operatorname{Aut}} \nc{\Inn}{\operatorname{Inn}}
\nc{\Der}{\operatorname{Der}} \nc{\Ker}{\operatorname{Ker}}
\nc{\Iso}{\operatorname{I}} \nc{\Diff}{\operatorname{Diff}}
\nc{\Lie}{\operatorname{L}} \nc{\tr}{\operatorname{tr}} \nc{\dif}{\operatorname{d}}
\nc{\sen}{\operatorname{sen}} \nc{\modu}{\operatorname{mod}}
\nc{\CRic}{\operatorname{PP}} \nc{\Cric}{\operatorname{P}} \nc{\Ricci}{\operatorname{Ric}}
\nc{\sym}{\operatorname{sym}} \nc{\symac}{\operatorname{sym^{ac}}}
\nc{\symc}{\operatorname{sym^{c}}} \nc{\scalar}{\operatorname{sc}}
\nc{\grad}{\operatorname{grad}} \nc{\ricci}{\operatorname{Rc}}
\nc{\Nor}{\operatorname{Norm}}  \nc{\ricc}{\operatorname{Rc^{c}}} \nc{\Ricc}{\operatorname{Ric^{c}}} \nc{\ricac}{\operatorname{Rc^{ac}}} \nc{\Ricac}{\operatorname{Ric^{ac}}} \nc{\Riem}{\operatorname{Rm}}
\nc{\riccig}{\operatorname{ric^{\gamma}}} \nc{\Rin}{\operatorname{M}}
\nc{\Le}{\operatorname{L}} \nc{\tang}{\operatorname{T}}
\nc{\level}{\operatorname{level}} \nc{\rad}{\operatorname{r}}
\nc{\abel}{\operatorname{ab}} \nc{\CH}{\operatorname{CH}}
\nc{\mcc}{\operatorname{mcc}} \nc{\Adj}{\operatorname{Adj}}
\nc{\Order}{\operatorname{O}}  \nc{\inj}{\operatorname{inj}} \nc{\proy}{\operatorname{pr}}
\nc{\vol}{\operatorname{vol}} \nc{\Diag}{\operatorname{Diag}}
\nc{\spa}{\operatorname{span}}

\theoremstyle{plain}
\newtheorem{theorem}{Theorem}[section]
\newtheorem{proposition}[theorem]{Proposition}
\newtheorem{corollary}[theorem]{Corollary}
\newtheorem{lemma}[theorem]{Lemma}

\theoremstyle{definition}
\newtheorem{definition}[theorem]{Definition}

\theoremstyle{remark}
\newtheorem{remark}[theorem]{Remark}

\newtheorem{example}[theorem]{Example}

\title[]{On the Chern-Ricci flow and its solitons for Lie groups}

\author{Jorge Lauret} \author{Edwin Alejandro Rodr\'\i guez Valencia}

\address{Universidad Nacional de C\'ordoba, FaMAF and CIEM, 5000 C\'ordoba, Argentina}
\email{lauret@famaf.unc.edu.ar}  \email{earodriguez@famaf.unc.edu.ar}

\thanks{This research was partially supported by grants from CONICET, FONCYT and SeCyT (Univ. Nac. C\'ordoba)}

\begin{document}

\maketitle

\begin{abstract}
This paper is concerned with Chern-Ricci flow evolution of left-invariant hermitian structures on Lie groups.  We study the behavior of a solution, as $t$ is approaching the first time singularity, by rescaling in order to prevent collapsing and obtain convergence in the pointed (or Cheeger-Gromov) sense to a Chern-Ricci soliton.  We give some results on the Chern-Ricci form and the Lie group structure of the pointed limit in terms of the starting hermitian metric and, as an application, we obtain a complete picture for the class of solvable Lie groups having a codimension one normal abelian subgroup.  We have also found a Chern-Ricci soliton hermitian metric on most of the complex surfaces which are solvmanifolds, including an unexpected shrinking soliton example.
\end{abstract}

\tableofcontents

\section{Introduction}\label{intro}

The {\it Chern-Ricci flow} (CRF) is the evolution equation for a one-parameter family $\omega(t)$ of hermitian metrics on a fixed complex manifold $(M,J)$ defined by
\begin{equation}\label{CRF-intro}
\dpar \omega=-2p, \qquad\mbox{or equivalently}, \qquad \dpar g=-2p(\cdot,J\cdot),
\end{equation}
where $p=p(J,\omega(t))$ is the Chern-Ricci form and $g=\omega(\cdot,J\cdot)$ (see \cite{Gll,TstWnk,TstWnk2}).  This paper is concerned with CRF-flow evolution of (compact) hermitian manifolds $(M,J,\omega)$ whose universal cover is a Lie group $G$ and such that if $\pi:G\longrightarrow M$ is the covering map, then $\pi^*J$ and $\pi^*\omega$ are left-invariant.  This is in particular the case of invariant structures on a quotient $M=G/\Gamma$, where $\Gamma$ is a cocompact discrete subgroup of $G$ (e.g. solvmanifolds and nilmanifolds).  A CRF-flow solution on $M$ is obtained by pulling down the corresponding CRF-flow solution on the Lie group $G$, which by diffeomorphism invariance stays left-invariant.  Equation \eqref{CRF-intro} therefore becomes an ODE for a non-degenerate $2$-form $\omega(t)$ on the Lie algebra $\ggo$ of $G$ and thus short-time existence (forward and backward) and uniqueness of the solutions are always guaranteed (see \cite{SCF}).

Let $(G,J)$ be a Lie group endowed with a left-invariant complex structure. Since on Lie groups the Chern-Ricci form $p$ depends only on $J$ (see \eqref{CRform}), we obtain that along the CRF-solution starting at a left-invariant hermitian metric $\omega_0$,  $p(t)\equiv p_0:=p(J,\omega_0)$.  This implies that $\omega(t)$ is simply given by
$$
\omega(t)=\omega_0-2tp_0.
$$
If $P_0$ is the Chern-Ricci operator of $\omega_0$ (i.e. $p_0=\omega_0(P_0\cdot,\cdot)$), then
$$
\omega(t)=\omega_0((I-2tP_0)\cdot,\cdot),
$$
and so the solution exists as long as the hermitian map $I-2tP_0$ is positive, say on a maximal interval $(T_-,T_+)$, which can be easily computed in terms of the extremal eigenvalues of the symmetric operator $P_0$.

We aim to understand the behavior of a CRF-solution $(G,\omega(t))$, as $t$ is approaching $T_\pm$, in the same spirit as in \cite[Section 3]{Ltt}, where the long-time behavior of homogeneous type-III Ricci flow solutions is studied.  In order to prevent collapsing and obtain an equally dimensional limit, the question is whether we can find a hermitian manifold $(M,J_\pm,\omega_\pm)$, bi-holomorphic diffeomorphisms $\phi(t):M\longrightarrow G$ and a scaling function $a(t)>0$ so that $a(t)\phi(t)^*\omega(t)$ converges smoothly to $\omega_\pm$, as $t\to T_\pm$.  Sometimes it is only possible to obtain this along a subsequence $t_k\to T_\pm$ and the diffeomorphisms $\phi(t_k)$ may be only defined on open subsets $\Omega_k$ exhausting $M$ and so $M$ might be non-diffeomorphic and even non-homeomorphic to $G$. This is called {\it pointed} or {\it Cheeger-Gromov} convergence of $(G,a(t)\omega(t))$ toward $(G_\pm,\omega_\pm)$.

It is proved in \cite{SCF} that given any CRF-solution $(G,\omega(t))$, there is always a pointed limit $(G_\pm,J_\pm,\omega_\pm)$ as above, where $G_\pm$ is a Lie group (possibly non-isomorphic to $G$) and the hermitian structure $(J_\pm,\omega_\pm)$ is left-invariant.  Moreover, $(J_\pm,\omega_\pm)$ is a {\it CR-soliton}, i.e.
$$
p(J_\pm,\omega_\pm)=c\omega_\pm+\lca_{X}\omega_\pm,
$$
for some $c\in\RR$ and a complete holomorphic vector field $X$ on $G_\pm$, or equivalently, the CRF-flow solution $\widetilde{\omega}(t)$ starting at $\omega_\pm$ is self-similar, in the sense that
$$
\widetilde{\omega}(t)=(-2ct+1)\vp(t)^*\omega_\pm,
$$
for some bi-holomorphic diffeomorphisms $\vp(t)$ of $(G_\pm,J_\pm)$.  Actually, $\vp(t)$ can be chosen to be a one-parameter group of automorphisms of $G_\pm$.  In many cases, the rescaling considered to obtain a pointed limit is the usual one given by $\omega(t)/t$.

After some preliminaries, we prove in Section \ref{chern} that any hermitian nilmanifold (i.e. $G$ nilpotent) is Chern-Ricci flat and so a fixed point for CRF.  In Sections \ref{CRF-sec} and \ref{CRsol-sec}, we overview the bracket flow approach and a structural result on CR-solitons from \cite{SCF} and then give a construction procedure for CR-solitons, including a characterization of those which are K\"ahler-Ricci solitons.

We study in Section \ref{CR-conv-sec} until what extent is the Chern-Ricci form and the Lie group structure of the pointed limit $(G_\pm,\omega_\pm)$ determined by the starting hermitian metric $(G,\omega_0)$.  For instance, we proved the following:

\begin{itemize}
\item If $P_0\leq 0$ (i.e. $T_+=\infty$) and $\kg:=\Ker P_0$ is an abelian ideal of $\ggo$, then $\omega(t)/t$ converges in the pointed sense, as $t\to\infty$, to a Chern-Ricci soliton $(G_+,\omega_+)$ with Lie algebra $\ggo_+=\kg^\perp\ltimes\kg$ such that $[\kg,\kg]_+=0$ and with Chern-Ricci operator given by $P_+|_{\kg^\perp}=-I$, $P_+|_{\kg}=0$.

\item If the eigenspace $\ggo_m$ of the maximum positive eigenvalue of $P_0$ is a nonzero Lie subalgebra of $\ggo$, then $T_+<\infty$ and $\omega(t)/(T_+-t)$ converges in the pointed sense, as $t\to T_+$, to a Chern-Ricci soliton $(G_+,\omega_+)$ with Lie algebra $\ggo_+=\ggo_m\ltimes\ggo_m^\perp$ such that $[\ggo_m^\perp,\ggo_m^\perp]_+=0$ and with Chern-Ricci operator given by $P_+|_{\ggo_m}=\unm I$, $P_+|_{\ggo_m^\perp}=0$.
\end{itemize}

In Section \ref{muA-sec}, we apply the above mentioned results on convergence and CR-solitons to the class of solvable Lie groups having a codimension one normal abelian subgroup.

Finally, we deal with complex surfaces in Section \ref{4dim}.  The family of $4$-dimensional solvable Lie groups admitting a left-invariant complex structure is quite large.  It consists of $19$ groups, although six of them are actually continuous pairwise non-isomorphic families (see Table \ref{tablalgsol}).  Moreover, many of them admit more than one complex structure up to equivalence and one of them does admit a two-parameter continuous family of complex structures (see Table \ref{tabcplxstruc}).  This classification was obtained in \cite{Ovn}.  We have found a CR-soliton hermitian metric for each of these complex structures, with the exceptions of only seven structures.  Most of them are either expanding or steady (i.e. $c\leq 0$), but one of the groups does admit an unexpected shrinking (i.e. $c>0$) CR-soliton (see Example \ref{exgama}).  We were able to prove the non-existence of a CR-soliton in only one of the seven cases, in which we furthermore found the non-isomorphic CR-soliton $(G_+,\omega_+)$ where all CRF-solutions are converging to (see Example \ref{r41}).  The CR-soliton metrics and their respective Chern-Ricci operators are given in Table \ref{tabsoliton}.

\vs \noindent {\it Acknowledgements.} We are very grateful to Isabel Dotti for pointing us to the reference \cite{BrbDttVrb} and for helpful conversations.

\section{Chern-Ricci form}\label{chern}

Let $(M,J,\omega,g)$ be a $2n$-dimensional hermitian manifold, where $\omega=g(J\cdot,\cdot)$.  The {\it Chern connection} is the unique connection $\nabla$ on $M$ which is hermitian (i.e. $\nabla J=0$, $\nabla g=0$) and its torsion satisfies $T^{1,1}=0$.  In terms of the Levi Civita connection $D$ of $g$, the Chern connection is given by
$$
g(\nabla_XY,Z)=g(D_XY,Z)-\unm d\omega(JX,Y,Z).
$$
We refer to e.g. \cite[(2.1)]{Vzz2}, \cite[(2.1)]{DsCVzz} and \cite[Section 2]{TstWnk} for different equivalent descriptions.  Note that $\nabla=D$ if and only if $(M,J,\omega,g)$ is K\"ahler.

The {\it Chern-Ricci form} $p=p(J,\omega,g)$ is defined by
$$
p(X,Y)=\sum_{i=1}^{n} g(R(X,Y)e_i,Je_i),
$$
where $R(X,Y)=\nabla_{[X,Y]} - [\nabla_X,\nabla_Y]$ is the curvature tensor of $\nabla$ and $\{ e_i,Je_i\}_{i=1}^n$ is a local orthonormal frame for $g$.  It follows that $p$ is closed, of type $(1,1)$ (i.e. $p=p(J\cdot,J\cdot)$), locally exact and in the K\"ahler case coincides with the Ricci form $\ricci(J\cdot,\cdot)$.

Consider now a left-invariant (almost-) hermitian structure $(J,\omega,g)$ on a Lie group with Lie algebra $\ggo$.  The integrability condition can be written as \begin{equation}\label{intJ}
[JX,JY]=[X,Y]+J[JX,Y]+J[X,JY],\qquad\forall X,Y\in\ggo.
\end{equation}
It is proved in \cite[Proposition 4.1]{Vzz2} (see also \cite{Pk}) that the Chern-Ricci form of $(J,\omega,g)$ is given by
\begin{equation}\label{CRform}
p(X,Y)=-\unm\tr{J\ad{[X,Y]}} + \unm\tr{\ad{J[X,Y]}}, \qquad\forall X,Y\in\ggo.
\end{equation}
We note that, remarkably, $p$ only depends on $J$.  The {\it Chern-Ricci operator} $P\in\End(\ggo)$, defined by
\begin{equation}\label{CRop}
p=\omega(P\cdot,\cdot),
\end{equation}
is a symmetric and hermitian map with respect to $(J,g)$ which vanishes on the center of $\ggo$.

It follows from \cite[Proposition 4.2]{Vzz2} that $p$ vanishes if $J$ is {\it bi-invariant} (i.e. $[J\cdot,\cdot]=J[\cdot,\cdot]$) or $J$ is {\it abelian} (i.e. $[J\cdot,J\cdot]=[\cdot,\cdot]$) and $\ggo$ unimodular.  On the other hand, it follows from \cite[Lemma 2.2]{BrbDttVrb} that hermitian nilmanifolds are all Chern-Ricci flat.  We now give a proof of this fact for completeness, which is based on the proof of such lemma and it is a bit shorter.

\begin{proposition}\label{Pnilp}
The Chern-Ricci form vanishes for any left-invariant hermitian structure on a nilpotent Lie group.
\end{proposition}

\begin{proof}
It is sufficient to prove that $\tr(J\ad_{X})=0$ for any $X\in\ggo$ (see \eqref{CRform}), or
equivalently, $\tr(J^c\ad_{X})=0$, for any $X\in\ggo_{\CC}$, where $\ggo_{\CC}=\ggo\oplus i\ggo$ is the complexification of $\ggo$ and
$J^c:\ggo_{\CC}\to\ggo_{\CC}$ is given by $J^c(X+iY)=JX + iJY$.  Consider now the decomposition $\ggo_{\CC}=\ggo^{1,0}\oplus\ggo^{0,1}$ in $\pm i$-eigenspaces of $J^c$.  Since $J$ is integrable
and $\ggo$ is nilpotent, we have that $\ggo^{1,0}$ is a (complex) nilpotent Lie subalgebra of
$\ggo_{\CC}$. It follows that if $\{X_1,\ldots,X_n\}$ is a basis of $\ggo^{1,0}$, then $\beta=\{X_1,\ldots,X_n,\overline{X}_1,\ldots,\overline{X}_n\}$
is a basis of $\ggo_{\CC}$ and the matrix of $\ad_{X_k}$ relative to $\beta$ has the form
$$
\left[\begin{matrix} A_k& \ast\\ 0& B_k
\end{matrix}\right].
$$
Since $\tr{A_k}=0$ and $\tr{\ad_{X_k}}=0$ by nilpotency, we obtain that $\tr{B_k}=0$.  On the other hand, as
the matrix of $J^c$ relative to $\beta$ is given by
$$
\left[\begin{matrix} iId& 0\\ 0& -iId
\end{matrix}\right],
$$
it follows that the matrix of $J^c\ad_{X_k}$ is of the form
$$
\left[\begin{matrix} iA_k& \ast\\ 0& -iB_k
\end{matrix}\right],
$$
and so it has zero trace. A similar argument gives that
$\tr(J^c\ad_{\overline{X}_k})=0$, concluding the proof.
\end{proof}

\section{Chern-Ricci flow}\label{CRF-sec}

Let $(M,J)$ be a complex manifold.  The {\it Chern-Ricci flow} (CRF) is the evolution equation for a one-parameter family $\omega(t)$ of hermitian metrics defined by
\begin{equation}\label{CRF}
\dpar \omega=-2p, \qquad\mbox{or equivalently}, \qquad \dpar g=-2p(\cdot,J\cdot),
\end{equation}
where $p=p(J,\omega(t))$ is the Chern-Ricci form and $g=\omega(\cdot,J\cdot)$.  We refer to \cite{Gll,TstWnk,TstWnk2} and the references therein for further information on this flow.  If the starting metric $\omega_0$ is K\"ahler, then CRF becomes the K\"ahler-Ricci flow (KRF).

Let $(G,J)$ be a Lie group endowed with a left-invariant complex structure.  Given a left-invariant hermitian metric $\omega_0$, it follows from the  diffeomorphism invariance of equation \eqref{CRF} that the CRF-solution starting at $\omega_0$ stays left-invariant and so it can be studied on the Lie algebra.  Indeed, the CRF becomes the ODE system
\begin{equation}\label{eqLG}
\ddt\omega=-2p,
\end{equation}
where $\omega(t),p(t)\in\Lambda^2\ggo^*$, as all the tensors involved are determined by their value at the identity of the group.  Thus short-time existence (forward and backward) and uniqueness of the solutions are always guaranteed.

Since on Lie groups the Chern-Ricci form $p$ depends only on $J$ (see \eqref{CRform}), we obtain that along the CRF-solution starting at $\omega_0$, $p(t)\equiv p_0:=p(J,\omega_0)$, and so $\omega(t)$ is simply given by
\begin{equation}\label{CRFsol}
\omega(t)=\omega_0-2tp_0, \qquad\mbox{or equivalently}, \qquad g(t)=g_0-2tp_0(\cdot,J\cdot).
\end{equation}
If $P_0$ is the Chern-Ricci operator of $\omega_0$ (see \eqref{CRop}), then
$$
\omega(t)=\omega_0((I-2tP_0)\cdot,\cdot),
$$
and so the solution exists as long as the hermitian map $I-2tP_0$ is positive.  It follows that the maximal interval of time existence $(T_-,T_+)$ of $\omega(t)$ is given by
\begin{equation}\label{CRFint}
T_+=\left\{\begin{array}{lcl} \infty, & \quad\mbox{if}\;P_0\leq 0, \\ \\ 1/(2p_+), & \quad\mbox{otherwise,}\end{array}\right. \qquad
T_-=\left\{\begin{array}{lcl} -\infty, & \quad\mbox{if}\;P_0\geq 0, \\ \\ 1/(2p_-), & \quad\mbox{otherwise,}\end{array}\right.
\end{equation}
where $p_+$ is the maximum positive eigenvalue of the Chern-Ricci operator $P_0$ of $\omega_0$ (see \eqref{CRop}) and $p_-$ is the minimum negative eigenvalue.

\subsection*{Bracket flow}
Given a left-invariant hermitian metric $\omega_0$ on a simply connected Lie group $(G,J)$ endowed with a left-invariant complex structure, one has that the new metric
$$
\omega=h^*\omega_0:=\omega_0(h\cdot,h\cdot),
$$
is also hermitian for any $h\in\Gl(\ggo,J)\simeq\Gl_n(\CC)$.  Moreover, the corresponding holomorphic Lie group isomorphism
$$
\widetilde{h}:(G,J,\omega)\longrightarrow (G_\mu,J,\omega_0), \qquad\mbox{where}\qquad \mu=h\cdot\lb:=h[h^{-1}\cdot,h^{-1}\cdot],
$$
is an equivalence of hermitian manifolds.  Here $\lb$ denotes the Lie bracket of the Lie algebra $\ggo$ and so $\mu$ defines a new Lie algebra (isomorphic to $(\ggo,\lb)$) with same underlying vector space $\ggo$.  We denote by $G_\mu$ the simply connected Lie group with
Lie algebra $(\ggo,\mu)$.  This equivalence suggests the following natural question:

\begin{quote}
What if we evolved $\mu$ rather than $\omega$?
\end{quote}

We consider for a family $\mu(t)\in \Lambda^2\ggo^*\otimes\ggo$ of Lie brackets the following evolution equation:
\begin{equation}\label{BF}
\ddt\mu=\delta_\mu(P_\mu), \qquad\mu(0)=\lb,
\end{equation}
where $P_\mu\in\End(\ggo)$ is the Chern-Ricci operator of the hermitian manifold $(G_\mu,J,\omega_0)$ and $\delta_\mu:\End(\ggo)\longrightarrow\Lambda^2\ggo^*\otimes\ggo$ is defined by
$$
\delta_\mu(A):=\mu(A\cdot,\cdot)+\mu(\cdot,A\cdot)-A\mu(\cdot,\cdot) = -\ddt|_{t=0} e^{tA}\cdot\mu, \qquad\forall A\in\End(\ggo).
$$
This evolution equation is called the {\it bracket flow} and has been proved in \cite{SCF} to be equivalent to the CRF.  Note that since $J$ is fixed, the algebraic subset
$$
\left\{\mu\in \Lambda^2\ggo^*\otimes\ggo: \mu\;\mbox{satisfies the Jacobi identity and}\; J\;\mbox{is integrable on}\; G_\mu\right\},
$$
is invariant under the bracket flow; indeed, $\mu(t)\in\Gl_n(\CC)\cdot\lb$ for all $t$.

For a given simply connected hermitian Lie group $(G,J,\omega_0)$ with Lie algebra $\ggo$, we may therefore consider the following two one-parameter families of hermitian Lie groups:
\begin{equation}\label{3rm}
(G,J,\omega(t)), \qquad (G_{\mu(t)},J,\omega_0),
\end{equation}
where $\omega(t)$ is the CRF (\ref{eqLG}) starting at $\omega_0$ and $\mu(t)$ is the bracket flow (\ref{BF}) starting at the Lie bracket $\lb$ of $\ggo$.

\begin{theorem}\label{eqfl}\cite[Theorem 5.1]{SCF}
There exist time-dependent holomorphic Lie group isomorphisms $h(t):G\longrightarrow G_{\mu(t)}$ such that
$$
\omega(t)=h(t)^*\omega_0, \qquad\forall t,
$$
which can be chosen such that their derivatives at the identity, also denoted by $h=h(t)$ (in particular, $\mu(t)=h(t)\cdot\lb$), is the solution to any of the following systems of ODE's:
\begin{itemize}
\item[(i)] $\ddt h=-hP$, $\quad h(0)=I$.
\item[(ii)] $\ddt h=-P_\mu h$, $\quad h(0)=I$.
\end{itemize}
\end{theorem}

The maximal interval of time existence $(T_-,T_+)$ is therefore the same for both flows, as it is the behavior
of any kind of curvature along the flows.

It is easy to see that the Chern-Ricci operator of $(G,J,\omega(t))$ equals
$$
P(t)=(I-2tP_0)^{-1}P_0,
$$
from which it follows that the family $h(t)\in\Gl(\ggo)$ is given by $h(t)=(I-2tP_0)^{1/2}$.  The solution to the bracket flow is therefore given by
$$
\mu(t)=(I-2tP_0)^{1/2}\cdot\lb,
$$
and hence relative to any orthonormal basis $\{ e_1,\dots,e_{2n}\}$ of eigenvectors of $P_0$, say with eigenvalues $\{ p_1,\dots,p_{2n}\}$, the structure coefficients of $\mu(t)$ are
\begin{equation}\label{muijk}
\mu_{ij}^k(t)=\left(\frac{1-2tp_k}{(1-2tp_i)(1-2tp_j)}\right)^{1/2} c_{ij}^k,
\end{equation}
where $c_{ij}^k$ are the structure coefficients of the Lie bracket $\lb$ of $\ggo$ (i.e. $[e_i,e_j]=\sum c_{ij}^ke_k$).

The Chern scalar curvature is therefore given by
$$
\tr{P(t)}=\sum_{i=1}^{2n} \frac{p_i}{1-2tp_i}.
$$
Thus $\tr{P(t)}$ is strictly increasing unless $P(t)\equiv 0$ (i.e. $\omega(t)\equiv \omega_0$) and the integral of $\tr{P(t)}$ must blow up at a finite-time singularity $T_+<\infty$.  However, $\tr{P(t)}\leq \frac{C}{T_+-t}$ for some constant $C>0$, which is the claim of a well-known general conjecture for the K\"ahler-Ricci flow (see e.g. \cite[Conjecture 7.7]{SngWnk}).

\section{Chern-Ricci solitons}\label{CRsol-sec}

In this section, we deal with self-similar CRF-solutions on Lie groups.  It follows from Proposition \ref{Pnilp} that $p=0$ if $\ggo$ is nilpotent, and thus any left-invariant hermitian structure on a nilpotent Lie group (and consequently, on any compact nilmanifold) is a fixed point for the CRF.  However, we will show in Section \ref{4dim} that several $4$-dimensional solvable Lie groups do admit Chern-Ricci solitons which are not fixed points (i.e. $p\ne 0$), including the covers of Inoue surfaces.

\begin{definition}\label{CRS}\cite[(39)]{SCF}
$(G,J,\omega)$ is said to be a {\it Chern-Ricci soliton} (CR-soliton) if its Chern-Ricci operator satisfies
$$
P=cI+\unm(D+D^t), \qquad\mbox{for some}\quad c\in\RR, \quad D\in\Der(\ggo), \quad DJ=JD.
$$
\end{definition}

This is equivalent to have
$$
p(J,\omega)=c\omega+\unm(\omega(D\cdot,\cdot)+\omega(\cdot,D\cdot)) =c\omega-\unm\lca_{X_D}\omega,
$$
where $X_D$ is the vector field on the Lie group defined by the one-parameter subgroup of automorphisms $\vp_t$ with derivative $e^{tD}\in\Aut(\ggo)$ and $\lca_{X_D}$ denotes Lie derivative.  The CRF-solution starting at a CR-soliton $(G,J,\omega)$ is given by
\begin{equation}\label{evsol}
\omega(t)=(-2ct+1) \left(e^{s(t)D}\right)^*\omega,
\end{equation}
where $s(t):=\frac{\log(-2ct+1)}{-2c}$ if $c\ne 0$ and $s(t)=t$ when $c=0$.

The following structural result for Chern-Ricci solitons, which in particular holds for K\"ahler-Ricci solitons, provides a starting point for approaching the classification problem.

\begin{proposition}\label{CR-sol2}\cite[Proposition 8.2]{SCF}
Let $(G,J,\omega)$ be a hermitian Lie group with Lie algebra $\ggo$ and Chern-Ricci operator $P\ne 0$.  Then the following conditions are equivalent.

\begin{itemize}
\item[(i)] $\omega$ is a Chern-Ricci soliton with constant $c$.
\item[(ii)] $P=cI+D$, for some $D\in\Der(\ggo)$.
\item[(iii)] The eigenvalues of $P$ are all either equal to $0$ or $c$, the kernel $\kg=\Ker{P}$ is an abelian ideal of $\ggo$ and its orthogonal complement $\kg^\perp$ (i.e. the $c$-eigenspace of $P$) is a Lie subalgebra of $\ggo$ (in particular, $\ggo$ is the semidirect product $\ggo=\kg^\perp\ltimes\kg$ and $c$ is always nonzero).
\end{itemize}
\end{proposition}

The following corollary essentially follows from the observation that $J$ must leave $\kg^\perp$ and $\kg$ invariant, as it commutes with $P$.

\begin{corollary}
Any Chern-Ricci soliton can be constructed as $(\ggo=\ggo_1\ltimes\ggo_2,J,\omega)$, with $J=\left[\begin{smallmatrix} J_1&\\ &J_2\end{smallmatrix}\right]$, $\omega=\omega_1\oplus\omega_2$, from the following data:

\begin{itemize}
\item a hermitian Lie algebra $(\ggo_1,J_1,\omega_1)$;
\item a hermitian abelian Lie algebra $(\ggo_2,J_2,\omega_2)$;
\item and a representation $\theta:\ggo_1\longrightarrow\End(\ggo_2)$;
\end{itemize}

such that the following conditions hold:

\begin{itemize}
\item $[\theta(J_1X),J_2]=J_2[\theta(X),J_2]$, for all $X\in\ggo_1$;
\item the Chern-Ricci operator $P_1$ of $(\ggo_1,J_1,\omega_1)$ equals
$$
P_1=cI-P_\theta,
$$
where $P_\theta\in\End(\ggo_1)$ is defined by
$$
\omega_1(P_\theta X,Y)=-\unm\tr{J_2\theta([X,Y])}+\unm\tr{\theta(J_1[X,Y])}, \qquad\forall X,Y\in\ggo_1.
$$
\end{itemize}
The Chern-Ricci operator of $(\ggo,J,\omega)$ is given by $P|_{\ggo_1}=cI$, $P|_{\ggo_2}=0$.

Moreover, $(\ggo,J,\omega)$ is K\"ahler (and so a K\"ahler-Ricci soliton) if and only if $\omega_1$ is closed (i.e. $(\ggo_1,J_1,\omega_1)$ K\"ahler) and $\theta(\ggo_1)\subset\spg(\ggo_2,\omega_2)$ (i.e. $\theta(X)^t=J_2\theta(X)J_2$ for all $X\in\ggo_1$).
\end{corollary}

\begin{proof}
It is easy to check that the first condition which must hold is equivalent to $J$ being integrable, and the second one comes from the fact that the Chern-Ricci operator of $(\ggo,J,\omega)$ is given by
$
P=\left[\begin{smallmatrix} P_1+P_\theta&0\\ 0&0\end{smallmatrix}\right].
$
The last claim on K\"ahler easily follows from the closeness condition for $\omega$.
\end{proof}

\begin{example}
We therefore obtain a Chern-Ricci soliton from any hermitian Lie algebra $(\ggo_1,J_1,\omega_1)$ with $P_1=cI$ (i.e. $p_1=c\omega_1$) and a representation $\theta:\ggo_1\longrightarrow\slg(\ggo_2,J_2)$ (i.e. $\tr{\theta(X)}=0$ and $[\theta(X),J_2]=0$ for all $X\in\ggo_1$); note that $P_\theta=0$ under such conditions.  If in addition $(\ggo_1,J_1,\omega_1)$ is K\"ahler-Einstein and
$$
\theta(\ggo_1)\subset\slg(\ggo_2,J_2)\cap\spg(\ggo_2,\omega_2)=\sug(\dim{\ggo_2}/2),
$$
then what we obtain is a K\"ahler-Ricci soliton, which is actually isometric to the direct product $G_1\times\RR^{\dim{\ggo_2}}$.
\end{example}

\section{Convergence}\label{CR-conv-sec}

We study in this section the possible limits of bracket flow solutions under diverse rescalings.

If a rescaling $c(t)\mu(t)$, $c(t)\in\RR$, of a bracket flow solution converges to $\lambda$, as $t\to T_\pm$, and $\vp(t):G\longrightarrow G_{c(t)\mu(t)}$ is the isomorphism with derivative $\tfrac{1}{c(t)}h(t)$, where $h(t)$ is as in Theorem \ref{eqfl}, then it follows from \cite[Corollary 6.20]{spacehm} that (after possibly passing to a subsequence) the Riemannian manifolds $\left(G,\tfrac{1}{c(t)^2}\omega(t)\right)$ converge in the pointed (or Cheeger-Gromov) sense to $(G_\lambda,\omega_0)$, as $t\to T_\pm$.  We note that $G_\lambda$ may be non-isomorphic, and even non-homeomorphic, to $G$ (see \cite[Section 5.1]{SCF}).

Recall also that all the limits obtained by any of such rescalings are automatically CR-solitons (see \cite[Section 7.1]{SCF}).

Two rescalings will be considered, the one given by the bracket norm $\mu(t)/|\mu(t)|$, which always converges, and $|2t+1|^{1/2}\mu(t)$, which corresponds according to the observation above to the standard rescaling $\omega(t)/(2t+1)$ of the original CRF-solution in the forward case.  We note that $\omega(t)/(2t+1)$ is, up to reparametrization in time, the solution to the {\it normalized Chern-Ricci flow}
\begin{equation}\label{nCRF}
\dpar\widetilde{\omega} = -2p(\widetilde{\omega}) - 2\widetilde{\omega}, \qquad\widetilde{\omega}(0)=\omega_0,
\end{equation}
which is the one preserving the volume in the compact K\"ahler case and has also been used in the general hermitian case (see e.g. \cite[Theorem 1.7]{TstWnk} and \cite{Gll}).

Let $(G,J,\omega_0)$ be a hermitian Lie group with Lie algebra $\ggo$ and Chern-Ricci operator $P_0$.  A straightforward analysis using \eqref{muijk} gives that $\mu(t)$ converges as $t\to T_\pm$ if and only if $T_\pm=\pm\infty$ (i.e. $\pm P_0\leq 0$) and $\Ker{P_0}$ is a Lie subalgebra of $\ggo$.  Moreover, the following conditions are equivalent in the case $T_\pm=\pm\infty$:

\begin{itemize}
\item $\mu(t)\to 0$, as $t\to\pm\infty$.
\item $\Ker{P_0}$ is an abelian ideal of $\ggo$.
\item $|2t+1|^{1/2}\mu(t)$ converges as $t\to\pm\infty$.
\end{itemize}

\begin{remark}
Any statement as the above ones, involving the $\pm$ sign, must always be understood as two separate statements, one for the $+$ sign and the other for the $-$ sign.
\end{remark}

In the case $\pm T_\pm<\infty$, it follows that $|T_\pm-t|^{1/2}\mu(t)$ converges as $t\to T_\pm$ if and only if $\ggo_\pm$ is a Lie subalgebra of $\ggo$,  where $\ggo_\pm$ is the eigenspace of $P_0$ of eigenvalue $p_\pm$ (see \eqref{CRFint}).

\begin{lemma}
If $\mu(t)\to\lambda$, as $t\to\pm\infty$, then $(G_\lambda,J,\omega_0)$ is Chern-Ricci flat.
\end{lemma}

\begin{proof}
We have that $\lambda$ is a fixed point and so the solution starting at $\lambda$ is defined on $(-\infty,\infty)$, which implies that $P_\lambda=0$ by \eqref{CRFint}.
\end{proof}

We now explore in which way is the limit of the normalization $\mu(t)/|\mu(t)|$ related to the starting point $(G,J,\omega_0)$.  The norm $|\mu|$ of a Lie bracket will be defined in terms of the canonical inner product on $\Lambda^2\ggo^*\otimes\ggo$ given by
\begin{equation}\label{ipg}
\la\mu,\lambda\ra:=\sum g_0(\mu(e_i,e_j),\lambda(e_i,e_j)) =\sum\mu_{ij}^k\lambda_{ij}^k,
\end{equation}
where $\{ e_i\}$ is any orthonormal basis of $(\ggo,g_0)$ and $g_0=\omega_0(\cdot,J\cdot)$.  A natural inner product on $\End(\ggo)$ is also determined by $g_0$ by $\la A,B\ra:=\tr{AB^t}$.

\begin{proposition}\label{CR-conv}
Let $(G,J,\omega_0)$ be a hermitian Lie group with Lie algebra $\ggo$ and Chern-Ricci operator $P_0$, and let $\kg$, $\ggo_+$ and $\ggo_-$ denote the eigenspaces of $P_0$ of eigenvalues $0$, $p_+$ and $p_-$, respectively (see \eqref{CRFint}).

\begin{itemize}
\item[(i)] The normalized Chern-Ricci bracket flow $\mu(t)/|\mu(t)|$ always converges, as $t\to T_\pm$, to a nonabelian Lie bracket $\lambda_\pm$ such that $(G_{\lambda_\pm},J,\omega_0)$ is a Chern-Ricci soliton, say with Chern-Ricci operator $P_{\lambda_\pm}$.
\item[ ]
\item[(ii)] If $\pm P_0\leq 0$ (i.e. $\pm T_\pm=\infty$), then $P_{\lambda_\pm}|_{\kg^\perp}=c_\pm I$, $P_{\lambda_\pm}|_{\kg}=0$, with $\pm c_\pm<0$ if and only if $\kg$ is an abelian ideal of $\ggo$.  Otherwise, $P_{\lambda_\pm}=0$.
\item[ ]
\item[(iii)] In the case when $\pm T_\pm<\infty$, one has $P_{\lambda_\pm}|_{\ggo_\pm}=c_\pm I$, $P_{\lambda_\pm}|_{(\ggo_\pm)^\perp}=0$, with $\pm c_\pm >0$ if and only if $\ggo_\pm$ is a Lie subalgebra of $\ggo$.  Otherwise, $P_{\lambda_\pm}=0$.
\end{itemize}
\end{proposition}

\begin{remark}
In particular, the only way to obtain in the limit the Einstein-like condition $p_{\lambda_\pm}=c\omega_0$ with $c\ne 0$, is precisely when $\pm P_0<0$.
\end{remark}

\begin{remark}
It follows from the last paragraph in \cite[Section 5.1]{SCF} that when $P_{\lambda_\pm}=0$, it actually holds that $P_\nu=0$ for any limit $\nu=\lim\limits_{t\to T_\pm} c(t)\mu(t)$ and any rescaling of the form $c(t)\mu(t)$, with $c(t)\in\RR$.
\end{remark}

\begin{proof}
It follows from \eqref{ipg} and \eqref{muijk} that
\begin{equation}\label{convsc}
\frac{\mu_{rs}^l}{|\mu|}=\frac{c_{rs}^l}{\left(\sum\limits_{i,j,k} \frac{(1-2tp_k)(1-2tp_r)(1-2tp_s)}{(1-2tp_i)(1-2tp_j)(1-2tp_l)} \left(c_{ij}^k\right)^2\right)^{1/2}} \underset{t\to T_\pm}\longrightarrow (\lambda_\pm)_{rs}^l.
\end{equation}
Since each of the terms in the sum above converges, as $t\to T_\pm$, to either a nonnegative real number or $\infty$, we obtain that $\mu(t)/|\mu(t)|$ always converges, and so part (i) follows.

We will only prove the $+$-statements, the proofs for those with a $-$ sign are completely analogous.  Since $P_{\mu/|\mu|}=\tfrac{1}{|\mu|^2}P\to P_{\lambda_+}$, as $t\to T_+$ (recall that $P_{\mu(t)}=P(t)=(I-2tP_0)^{-1}P_0$), one can easily check that for each eigenvalue $p_r$ of $P_0$,
$$
\begin{array}{l}
\lim\limits_{t\to T_+}\frac{p_r}{|\mu|^2(1-2tp_r)} = \lim\limits_{t\to T_+}\frac{p_r}{\sum\limits_{i,j,k}\frac{(1-2tp_k)(1-2tp_r)}{(1-2tp_i)(1-2tp_j)}\left(c_{ij}^k\right)^2} \\ \\
= \left\{\begin{array}{ll}
\frac{1}{\sum\limits_{p_i,p_j,p_k<0}\frac{p_k}{p_ip_j}\left(c_{ij}^k\right)^2 +2\sum\limits_{p_i<0,p_j=p_k=0}\frac{1}{p_i}\left(c_{ij}^k\right)^2}<0, & T_+=\infty, \quad p_r<0, \quad\kg\; \mbox{abelian ideal}; \\ \\
0, & T_+=\infty, \quad \mbox{otherwise};  \\ \\
\frac{p_+}{\sum\limits_{i,j,k=+}\left(c_{ij}^k\right)^2 +2\sum\limits_{i,k\ne +, j=+}\frac{p_+-p_k}{p_+-p_i}\left(c_{ij}^k\right)^2}>0, & T_+<\infty, \quad p_r=p_+, \quad\ggo_+\; \mbox{subalgebra}; \\ \\
0, & T_+<\infty, \quad \mbox{otherwise}.
\end{array}\right.
\end{array}
$$
This shows that the value of $P_{\lambda_+}$ is as in parts (ii) and (iii), concluding the proof of the proposition.
\end{proof}

\begin{proposition}\label{CR-conv3}
Let $(G,J,\omega_0)$ be a hermitian Lie group as in the above proposition and consider $\lambda_\pm$, the limit of $\mu(t)/|\mu(t)|$ as $t\to T_\pm$.

\begin{itemize}
\item[(i)] If $\pm P_0\leq 0$ (i.e. $\pm T_\pm=\infty$) and $\kg$ is an abelian ideal of $\ggo$, then $(\ggo,\lambda_\pm)=\kg^\perp\ltimes\kg$ and $\lambda_\pm(\kg,\kg)=0$.  On the contrary, if $\kg$ is not an abelian ideal of $\ggo$, then $\kg^\perp$ is an abelian ideal of $(\ggo,\lambda_\pm)$.  Moreover, if $\kg$ is not even a Lie subalgebra of $\ggo$, then $\lambda_\pm$ is $2$-step nilpotent and $\kg^\perp$ is contained in its center.
\item[ ]
\item[(ii)] If $\pm T_\pm<\infty$ and $\ggo_\pm$ is a Lie subalgebra of $\ggo$, then $(\ggo,\lambda_\pm)=\ggo_\pm\ltimes\ggo_\pm^\perp$ and $\lambda_\pm(\ggo_\pm^\perp,\ggo_\pm^\perp)=0$.  On the contrary, if $\ggo_\pm$ is not a Lie subalgebra of $\ggo$, then $\lambda_\pm$ is $2$-step nilpotent and $\ggo_\pm^\perp$ is contained in its center.
\end{itemize}
\end{proposition}

\begin{proof}
The first claims in the items are both direct consequences of Proposition \ref{CR-sol2}, (iii).  As above, we only prove the $+$-statements.

If $\kg$ is not an abelian ideal of $\ggo$, then there is a $c_{ij}^k\ne 0$ with either $p_i,p_j,p_k=0$, or $p_ip_j=0$ and $p_k<0$.  The corresponding term in the sum appearing in formula \eqref{convsc} therefore converges to $\infty$ for any triple $(r,s,l)$ such that either $p_r,p_s,p_l<0$, or $p_l=0$, or $p_l<0$ and at least one of $p_r,p_s$ is negative.  This implies that $\lambda_{rs}^l=0$ for all such triples and hence $\lambda_+(\kg^\perp,\kg^\perp)=0$ and $\lambda_+(\ggo,\kg^\perp)\subset\kg^\perp$, respectively.

Assume now that $\kg$ is not a subalgebra of $\ggo$.  Thus there is a $c_{ij}^k\ne 0$ with $p_i,p_j=0$ and $p_k<0$.  The corresponding term in \eqref{convsc} therefore converges to $\infty$ for any triple $(r,s,l)$ such that either $p_l=0$, or $p_l<0$ and at least one of $p_r,p_s$ is negative.  This implies that $\lambda_+(\ggo,\ggo)\subset\kg^\perp$ and $\lambda_+(\ggo,\kg^\perp)=0$, respectively.  The second claim in part (i) therefore follows.

It only remains to prove the second claim in part (ii).  If $\ggo_+$ is not a subalgebra of $\ggo$, then there is a $c_{ij}^k\ne 0$ with $p_i,p_j=p_+$ and $p_l\ne p_+$.  Thus the corresponding term in \eqref{convsc} does not converge to $\infty$ if and only if $p_r=p_s=p_+$ and $p_l\ne p_+$, that is, the only part of $\lambda_+$ which survives is $\lambda_+:\ggo_+\times\ggo_+\longrightarrow\ggo_+^\perp$, as was to be shown.
\end{proof}

We now study the rescaling $|2t+1|^{1/2}\mu(t)$, or equivalently $\omega(t)/(2t+1)$, corresponding to the normalized CRF given in \eqref{nCRF}.  Recall that we always denote by $\lb$ the Lie bracket of the Lie algebra $\ggo$ of the Lie group $G$.

\begin{proposition}\label{CR-conv2}
Let $(G,J,\omega_0)$ be a hermitian Lie group as in the propositions above.

\begin{itemize}
\item[(i)] If $\pm P_0\leq 0$ (i.e. $\pm T_\pm=\infty$) and $\kg$ is an abelian ideal of $\ggo$, then $|2t+1|^{1/2}\mu(t)$ converges, as $t\to\pm\infty$, to a Chern-Ricci soliton $\nu_\pm$ such that $(\ggo,\nu_\pm)=\kg^\perp\ltimes\kg$, $\nu_\pm(\kg,\kg)=0$ and with Chern-Ricci operator given by $P_{\nu_\pm}|_{\kg^\perp}=\mp I$, $P_{\nu_\pm}|_{\kg}=0$.

\item[(ii)] If $\ggo^\pm$ is a nonzero Lie subalgebra of $\ggo$, then $\pm T_\pm<\infty$ and $|T_\pm-t|^{1/2}\mu(t)$ converges, as $t\to T_\pm$, to a Chern-Ricci soliton $\nu_\pm$ such that $(\ggo,\nu_\pm)=\ggo_\pm\ltimes\ggo_\pm^\perp$, $\nu_\pm(\ggo_\pm^\perp,\ggo_\pm^\perp)=0$ and with $P_{\nu_\pm}|_{\ggo^\pm}=\pm\unm I$, $P_{\nu_\pm}|_{(\ggo^\pm)^\perp}=0$.
\end{itemize}
\end{proposition}

\begin{proof}
One can prove this proposition in much the same way as Propositions \ref{CR-conv} and \ref{CR-conv3}, by using for the second statements that
$$
\left(|2t+1|^{1/2}\mu(t)\right)_{rs}^l= \left(\frac{|2t+1|(1-2tp_l)}{(1-2tp_r)(1-2tp_s)}\right)^{1/2} c_{rs}^l \underset{t\to T_\pm}\longrightarrow (\nu_\pm)_{rs}^l,
$$
and considering separately the cases $T_+=\infty$ and $T_+<\infty$.
\end{proof}

\section{Almost-abelian Lie groups}\label{muA-sec}

We apply in this section the results obtained above on CR-solitons and convergence on a class of solvable Lie algebras, which are very simple from the algebraic point of view but yet geometrically rich and exotic.

Let $(G,J,\omega)$ be a hermitian Lie group with Lie algebra $\ggo$ and assume that $\ggo$ has a codimension-one abelian ideal $\ngo$.  These Lie algebras are sometimes called {\it almost-abelian} in the literature (see e.g. \cite{CnsMcr}).    It is easy to see that there exists an orthonormal basis $\{ e_1,\dots,e_{2n}\}$ such that
$$
\ngo=\la e_1,\dots,e_{2n-1}\ra, \qquad \omega=e^1\wedge e^{2n}+\dots+e^n\wedge e^{n+1}, \qquad Je_i=e_{2n+1-i} \quad (1\leq i\leq n),
$$
where $\{ e^i\}$ denotes the dual basis.  It follows from \eqref{intJ} that $J$ is integrable if and only if
$$
-[e_1,Je_i]=[e_{2n},e_i]-J[e_1,e_i]+J[e_{2n},Je_i], \qquad \forall i=2,\dots,2n-1,
$$
and since the left-hand side and the middle term in the right-hand side both vanish, we obtain that $J$ is integrable if and only if
$\ad{e_{2n}}$ leaves the subspace $\la e_2,\dots, e_{2n-1}\ra$ invariant and commutes with the restriction of $J$ on such subspace.  The matrix of $\ad{e_{2n}}$ in terms of $\{ e_i\}$ is therefore given by
\begin{equation}\label{ade2n}
\ad{e_{2n}}=\left[\begin{array}{c|c|c}
c&0&0\\\hline
d_1&&\\
\vdots&\quad A\quad&0\\
d_{2n-2}&&\\\hline
0&0&0
\end{array}\right], \qquad A\in\glg_{n-1}(\CC).
\end{equation}
We call $\mu=\mu_{A,c,d_1,\dots,d_{2n-2}}$ the Lie bracket on $\ggo$ defined by \eqref{ade2n} and the condition that $\ngo$ is an abelian ideal.  It is easy to prove that two of these Lie algebras are isomorphic if and only if the corresponding adjoint maps $\ad{e_{2n}}|_{\ngo}$ are conjugate up to nonzero scaling.

\begin{lemma}\label{muAmet}
Any hermitian Lie algebra $(\ggo,J,\omega)$ with a codimension-one abelian ideal is equivalent to
$$
(\ggo,\mu_{A,c,d_1,\dots,d_{2n-2}},J,\omega), \qquad\mbox{for some}\quad A\in\glg_{n-1}(\CC),\quad c\geq 0, \quad d_i\in\RR.
$$
The Chern-Ricci form and operator of this structure are respectively given by
$$
p=-\unm c(2c+\tr{A}) e^1\wedge e^{2n}, \qquad P=-\unm c(2c+\tr{A})\left[\begin{array}{c|c|c}
1&0&0\\\hline
&&\\
0&\quad 0\quad&0\\
&&\\\hline
0&0&1
\end{array}\right].
$$
\end{lemma}

\begin{proof}
It only remains to prove the formula for the Chern-Ricci form.  We use formula \eqref{CRform} to compute $p$ as follows:
\begin{align*}
p(e_{2n},e_1)=& -\unm\tr{J\left(c\ad{e_1} + \sum_{i=1}^{2n-2} d_i\ad{e_{i+1}}\right)} +\unm\tr{\ad{J(ce_1)}}, \\
=& -\unm c\tr{J\ad{e_1}} +\unm c\tr{\ad{e_{2n}}}, \\
=&\unm c^2+\unm c(c+\tr{A}) = \unm c(2c+\tr{A}), \\ \\
p(e_{2n},e_i)=& -\unm\tr{J\ad{Ae_i}}+\unm\tr{\ad{JAe_i}} = 0+0=0, \quad\forall i=2,\dots,2n-1,
\end{align*}
concluding the proof of the lemma.
\end{proof}

\begin{table}
\centering
\renewcommand{\arraystretch}{1.6}
\begin{tabular}{|c|l|}\hline
$\mathbf{\ggo}$ & \textbf{Lie Bracket}\\ \hline\hline
$\rg\hg_3$ & $[e_1,e_2]=e_3$ \\ \hline
$\rg\rg_{3,0}$ & $[e_1,e_2]=e_2$ \\ \hline
$\rg\rg_{3,1}$ & $[e_1,e_2]=e_2$, $[e_1,e_3]=e_3$ \\ \hline
$\rg\rg'_{3,0}$ & $[e_1,e_2]=-e_3$, $[e_1,e_3]=e_2$ \\ \hline
$\rg\rg'_{3,\gamma}$ & $[e_1,e_2]=\gamma e_2-e_3$, $[e_1,e_3]=e_2+\gamma e_3$, $\quad\gamma>0$ \\ \hline
$\rg_2\rg_2$ & $[e_1,e_2]=e_2$, $[e_3,e_4]=e_4$ \\ \hline
$\rg'_2$ & $[e_1,e_3]=e_3$, $[e_1,e_4]=e_4$, $[e_2,e_3]=e_4$, $[e_2,e_4]=-e_3$ \\ \hline
$\rg_{4,1}$ & $[e_4,e_1]=e_1$, $[e_4,e_2]=e_2$, $[e_4,e_3]=e_2+e_3$ \\ \hline
$\rg_{4,\alpha,1}$ & $[e_4,e_1]=e_1$, $[e_4,e_2]=\alpha e_2$, $[e_4,e_3]=e_3$, $\quad -1<\alpha\leq 1$, $\alpha\neq 0$ \\ \hline
$\rg_{4,\alpha,\alpha}$ & $[e_4,e_1]=e_1$, $[e_4,e_2]=\alpha e_2$, $[e_4,e_3]=\alpha e_3$, $\quad -1\leq\alpha<1$, $\alpha\neq0$ \\ \hline
$\rg'_{4,\gamma,\delta}$ & $[e_4,e_1]=e_1$, $[e_4,e_2]=\gamma e_2-\delta e_3$, $[e_4,e_3]=\delta e_2+\gamma e_3$, $\gamma\in\RR$, $\quad\delta>0$ \\ \hline
$\dg_{4}$ & $[e_1,e_2]_1=e_3$, $[e_4,e_1]_1=e_1$, $[e_4,e_2]_1=-e_2$ \\ \cline{2-2}
& $[e_1,e_2]_2=e_3$, $[e_4,e_1]_2=e_1$, $[e_4,e_2]_2=-e_2+e_3$\\ \hline
$\dg_{4,1}$ & $[e_1,e_2]=e_3$, $[e_4,e_1]=e_1$, $[e_4,e_3]=e_3$ \\ \hline
$\dg_{4,\frac{1}{2}}$ & $[e_1,e_2]_1=e_3$, $[e_4,e_1]_1=\frac{1}{2}e_1$, $[e_4,e_2]_1=\frac{1}{2}e_2$, $[e_4,e_3]_1=e_3$ \\ \cline{2-2}
& $\lb_2=\lb_1$ \\ \cline{2-2}
& $[e_1,e_2]_3=e_3$, $[e_4,e_1]_3=e_1$, $[e_4,e_2]_3=e_2$, $[e_4,e_3]_3=2e_3$ \\ \hline
$\dg_{4,\lambda}$ & $[e_1,e_2]_1=\lambda e_3$, $[e_4,e_1]_1=\lambda e_1$, \\
& $[e_4,e_2]_1=(1-\lambda)e_2$, $[e_4,e_3]_1=e_3$, $\quad\frac{1}{2}<\lambda\neq 1$ \\ \cline{2-2}
& $[e_1,e_2]_2=(1-\lambda)e_3$, $[e_4,e_1]_2=\lambda e_1$, \\
& $[e_4,e_2]_2=(1-\lambda)e_2$, $[e_4,e_3]_2=e_3$, $\quad\frac{1}{2}<\lambda<1$\\ \cline{2-2}
& $[e_1,e_2]_3=(\lambda-1)e_3$, $[e_4,e_1]_3=\lambda e_1$, \\
& $[e_4,e_2]_3=(1-\lambda)e_2$, $[e_4,e_3]_3=e_3$, $\quad1<\lambda$\\ \hline
$\dg'_{4,0}$ & $[e_1,e_2]=e_3$, $[e_4,e_1]=-e_2$, $[e_4,e_2]=e_1$ \\ \hline
$\dg'_{4,\delta}$ & $[e_1,e_2]=e_3$, $[e_4,e_1]=\unm e_1 - \frac{1}{\delta} e_2$, \\
& $[e_4,e_2]=\frac{1}{\delta}e_1 + \unm e_2$, $[e_4,e_3]=e_3$, $\quad\delta>0$\\ \hline
$\hg_{4}$ & $[e_1,e_2]=e_3$, $[e_4,e_1]=e_1$, $[e_4,e_2]=\sqrt{10}e_1+e_2$, $[e_4,e_3]=2e_3$ \\ \hline
\end{tabular}
\vspace{0.3cm}
\caption{Solvable Lie algebras of dimension $4$ admitting a complex structure.}\label{tablalgsol}
\end{table}

It is proved in \cite{SCFdim4} that $(\ggo,\mu,J,\omega)$ is K\"ahler (i.e. $d\omega=0$) if and only if $d_i=0$ for all $i$ and $A\in\ug(n)$ (i.e. $A^t=-A$).  In such case, the metric is known to be isometric to $\RR H^2\times\RR^{2n-2}$, where $\RR H^2$ denotes the $2$-dimensional real hyperbolic space (see e.g. \cite[Proposition 2.5]{Hbr}).  On the other hand, it is easy to prove that $(\ggo,\mu,J,\omega)$ is bi-invariant if and only if $c=d_1=\dots=d_{2n-2}=0$, and abelian if and only if $A=0$.

\begin{proposition}\label{muA}
Let $(G_\mu,J,\omega)$ be the hermitian Lie group with $\mu=\mu_{A,c,d_i}$.
\begin{itemize}
\item[(i)] $(G_\mu,J,\omega)$ is a CR-soliton if and only if either $p=0$ or $p\ne 0$ and $d_i=0$ for all $i$.

\item[(ii)] The maximal interval of time existence of the CRF-solution $\omega(t)$ starting at $(G_\mu,J,\omega)$ is
$$
\left\{\begin{array}{lcl} (\tfrac{1}{e},\infty), && e<0, \\ (-\infty,\tfrac{1}{e}), && e>0, \\ (-\infty,\infty), && e=0, \end{array}\right., \qquad\mbox{where}\quad e:=-c(2c+\tr{A}).
$$

\item[(iii)] If $T_\pm=\pm\infty$ and $p\ne 0$ (i.e. $e\ne 0$), then the rescaled solution $\omega(t)/|2t+1|$ converges in the pointed sense, as $t\to\pm\infty$, to the CR-soliton $(G_\lambda,J,\omega)$, where $\lambda=\unm |e|^{1/2}\mu_{A,c,0}$.

\item[(iv)] If $\omega$ is not a CR-soliton, then, as $t$ approaches any finite-time singularity, $c(t)\omega(t)$ converges in the pointed sense to $(H_3\times\RR)\times\RR^{2n-4}$, where $H_3\times\RR$ is the universal cover of the Kodaria-Thurston manifold, for some rescaling $c(t)>0$.
\end{itemize}
\end{proposition}

\begin{remark}
Recall that in part (iii), if $\lambda\not\simeq\mu$, which never holds if $c$ is not an eigenvalue of $A$, then the limit is a left-invariant hermitian metric on a different Lie group (see Example \ref{r41}).
\end{remark}

\begin{proof}
Part (i) follows from Proposition \ref{CR-sol2} and Lemma \ref{muAmet}, by using that the image of any derivation must be contained in $\ngo$, and part (ii) follows from \eqref{CRFint}.  Since $\kg$ is always an abelian ideal, the limit $\nu_\pm$ from Proposition \ref{CR-conv2} equals $\mu_{A,c,0}$, up to a positive scaling, and so part (iii) holds.  On the other hand, $\ggo_\pm=\la e_1,e_{2n}\ra$ is never a Lie subalgebra if $d_i\ne 0$ for at least one $i$, in which case by Proposition \ref{CR-conv3}, $\lambda_\pm$ is $2$-step nilpotent with $\ggo_\pm^\perp$ contained in its center.  Thus $(\ggo,\lambda_\pm)$ is isomorphic to $\hg_3\oplus\RR^{2n-3}$, where $\hg_3$ denotes the $3$-dimensional Heisenberg algebra, from which part (iv) follows.
\end{proof}

\begin{table}
\centering
\renewcommand{\arraystretch}{1.6}
\begin{tabular}{|c|l|}\hline
$\mathbf{\ggo}$ & \textbf{Complex structures}\\ \hline\hline
$\rg\hg_3$ & $J e_1 = e_2$, $J e_3 = e_4$ \\ \hline
$\rg\rg_{3,0}$ & $J e_1 = e_2$, $J e_3 = e_4$ \\ \hline
$\rg\rg_{3,1}$ & $J e_1 = e_4$, $J e_3 = e_2$ \\ \hline
$\rg\rg'_{3,0}$ & $J e_1 = e_4$, $J e_2 = e_3$ \\ \hline
$\rg\rg'_{3,\gamma}$ & $J_1 e_1 = e_4$, $J_1 e_3 = e_2 - 2\gamma e_3$, \ $\gamma>0$
\vline \ $J_2 e_1 = e_4$, $J_2 e_3 = 2\gamma e_3 - e_2$, \ $\gamma>0$\\ \hline
$\rg_2\rg_2$ & $J e_1 = e_2$, $J e_3 = e_4$ \\ \hline
$\rg'_2$ & $J_1 e_1 = e_3$, $J_1 e_2 = e_4$ \vline \
$J_{s,t} e_2 = -\frac{1}{t} e_1 - \frac{s}{t} e_2$, $J_{s,t} e_3 = e_4$,
\ $s\in\RR$, $t\neq 0$\\ \hline
$\rg_{4,1}$ & $J e_1 = e_2$, $J e_4 = e_3$ \\ \hline
$\rg_{4,\alpha,1}$ & $J e_1 = e_3$, $J e_4 = e_2$ \\ \hline
$\rg_{4,\alpha,\alpha}$ & $J e_4 = e_1$, $J e_2 = e_3$ \\ \hline
$\rg'_{4,\gamma,\delta}$ & $J_1 e_4 = e_1$, $J_1 e_2 = e_3$ \vline \
$J_2 e_4 = e_1$, $J_2 e_3 = e_2$ \\ \hline
$\dg_{4}$ & $J_1 e_3 = e_1$, $J_1 e_4 = e_2$ \vline \ $J_2=J_1$ \\ \hline
$\dg_{4,1}$ & $J e_1 = e_4$, $J e_2 = e_3$ \\ \hline
$\dg_{4,\frac{1}{2}}$ & $J_1 e_1 = e_2$, $J_1 e_4 = e_3$ \vline \  $J_2 e_2 = e_1$, $J_2 e_4 = e_3$ \vline \
$J_3 e_4 = e_1$, $J_3 e_3 = e_2$\\ \hline
$\dg_{4,\lambda}$ & $J_1 e_1 = e_4$, $J_1 e_2 = e_3$ \vline \ $J_2 e_1 = e_3$, $J_2 e_4 = e_2$ \vline \
$J_3 e_1 = e_3$, $J_3 e_2 = e_4$ \\ \hline
$\dg'_{4,0}$ & $J_1 e_1 = e_2$, $J_1 e_3 = e_4$ \vline \ $J_2 e_1 = e_2$, $J_2 e_4 = e_3$ \\ \hline
$\dg'_{4,\delta}$ & \mbox{\small$J_1 e_2 = e_1$, $J_1 e_4 = e_3$} \vline \
\mbox{\small$J_2 e_1 = e_2$, $J_2 e_3 = e_4$} \vline \
\mbox{\small$J_3 e_1 = e_2$, $J_3 e_4 = e_3$} \vline \ \mbox{\small$J_4 e_2 = e_1$, $J_4 e_3 = e_4$} \\ \hline
$\hg_{4}$ & $J e_1 = e_3$, $J e_4 = e_2$ \\ \hline
\end{tabular}
\vspace{0.3cm}
\caption{Complex structures on $4$-dimensional solvable Lie algebras.}\label{tabcplxstruc}
\end{table}

\section{Lie groups of dimension $4$}\label{4dim}

We now study the existence problem for CR-solitons on $4$-dimensional solvable Lie groups.  We have listed in Table \ref{tablalgsol} all $4$-dimensional solvable Lie algebras admitting a complex structure and in Table \ref{tabcplxstruc} all the complex structures up to equivalence on
each Lie algebra (see \cite{Ovn}).  In order to obtain simpler forms for the matrices of the complex structures and the CR-soliton metrics, we decided to give more than one different (but isomorphic) Lie brackets $\lb_i$ for each Lie algebra $\dg_4$, $\dg_{4,\lambda}$ with $\lambda\ne 1$, in such a way that the pair $(\lb_i,J_i)$ is integrable for any $i=1,2,3$ (see \eqref{intJ}).

Let $(G,J,\omega)$ be a $4$-dimensional hermitian Lie group with Lie algebra $\ggo$.

{\footnotesize

\begin{table}
\centering
\renewcommand{\arraystretch}{1.6}
\begin{tabular}{|c|c|c|c|c|c|c|}\hline
$\mathbf{\ggo}$ & $\mathbf{J}$ & \textbf{Metric} & $\mathbf{P}$ & $\mathbf{c}$ & $\mathbf{D}$ & \textbf{K}\\ \hline\hline
$\rg\hg_3$ & $J$ & Any & $(0, 0, 0, 0)$ & 0 & $(0, 0, 0, 0)$ & ------\\ \hline
$\rg\rg_{3,0}$ & $J$ & $(1,1,1,1)$ & $(-1, -1, 0, 0)$ & $-1$ & $(0, 0, 1, 1)$ & Yes\\ \hline
$\rg\rg_{3,1}$ & $J$ & Any & $(0, 0, 0, 0)$ & 0 & $(0, 0, 0, 0)$ & ------\\ \hline
$\rg\rg'_{3,0}$ & $J$ & Any & $(0, 0, 0, 0)$ & 0 & $(0, 0, 0, 0)$ & Yes\\ \hline
$\rg\rg'_{3,\gamma}$ & $J_1$ & Any & $(0, 0, 0, 0)$ & 0 & $(0, 0, 0, 0)$ & ------\\ \cline{2-7}
& $J_2$ & Any & $(0, 0, 0, 0)$ & 0 & $(0, 0, 0, 0)$ & ------\\ \hline
$\rg_2\rg_2$ & $J$ & $(1,1,1,1)$ & $(-1, -1, -1, -1)$ & $-1$ & $(0, 0, 0, 0)$ & Yes\\ \hline
$\rg'_2$ & $J_1$ & $(1,1,1,1)$ & $(-2, 2, -2, 2)$ & ------ & ------ & ------\\ \cline{2-7}
& $J_{s,t}$ & Any & $(0, 0, 0, 0)$ & 0 & $(0, 0, 0, 0)$ & ------\\ \hline
$\rg_{4,1}$ & $J$ & $(1,1,1,1)$ & $(0, 0, -2, -2)$ & ------ & ------ & ------\\ \hline
$\rg_{4,\alpha,1}$ & $J$ & $(1,1,1,1)$ & $-\alpha(\alpha+1)(0, 1, 0, 1)$ &
$-\alpha(\alpha+1)$ & $\alpha(\alpha+1)(1, 0, 1, 0)$ & ------\\
& & & $-1<\alpha\leq 1$, $\alpha\neq 0$ & & & \\ \hline
$\rg_{4,\alpha,\alpha}$ & $J$ & $(1,1,1,1)$ & $-(\alpha+1)(1, 0, 0, 1)$ &
$-(\alpha+1)$ & $(\alpha+1)(0, 1, 1, 0)$ & ------\\
& & & $-1\leq\alpha<1$, $\alpha\neq 0$ & & & \\ \hline
$\rg'_{4,\gamma,\delta}$ & $J_1$ & $(1,1,1,1)$ & $-(\gamma+1)(1, 0, 0, 1)$ &
$-(\gamma+1)$ & $(\gamma+1)(0, 1, 1, 0)$ & $\gamma=0$\\
& & & $\gamma\in\RR$, $\delta>0$ & & & \\ \cline{2-7}
& $J_2$ & $(1,1,1,1)$ & $-(\gamma+1)(1, 0, 0, 1)$ &
$-(\gamma+1)$ & $(\gamma+1)(0, 1, 1, 0)$ & $\gamma=0$\\
& & & $\gamma\in\RR$, $\delta>0$ & & & \\ \hline
$\dg_{4}$ & $J_1$ & $(1,1,1,1)$ & $(0, -1, 0, -1)$ & $-1$ & $(1, 0, 1, 0)$ & ------\\ \cline{2-7}
& $J_2$ & $(1,1,1,1)$ & $(0, -1, 0, -1)$ & ------ & ------ & ------\\ \hline
$\dg_{4,1}$ & $J$ & $(1,1,1,1)$ & $(-2, 0, 0, -2)$ & $-2$ & $(0, 2, 2, 0)$ & ------\\ \hline
$\dg_{4,\frac{1}{2}}$ & $J_1$ & $(1,1,1,1)$ & $-\tfrac{3}{2}(1, 1, 1, 1)$ & $-\tfrac{3}{2}$ & $(0, 0, 0, 0)$ & Yes\\ \cline{2-7}
& $J_2$ & $(1,1,1,1)$ & $\tfrac{3}{2}(1, 1, -1, -1)$ & ------ & ------ & ------\\ \cline{2-7}
& $J_3$ & $(2,5,\frac{5}{4},2)$ & $-3(1, 0, 0, 1)$ & $-3$ &
$3(0, 1, 1, 0)$ &------\\ \hline
$\dg_{4,\lambda}$ & $J_1$ & $\left(\frac{e}{\lambda^2},
2,2,e\right)$ &
$(a, 0, 0, a)$ & $a$ & $(0, -a, -a, 0)$ & $\lambda=2$\\
& & & $\frac{1}{2}<\lambda\neq 1$ & & &\\ \cline{2-7}
& $J_2$ & $\left(2,\frac{d}{(\lambda-1)^2},2,d\right)$ &
$(0, b, 0, b)$ & $b$ & $(-b, 0, -b, 0)$ & ------\\
& & & $\unm < \lambda < 1$ & & & \\ \cline{2-7}
& $J_3$ & $\left(2,\frac{d}{(\lambda-1)^2},2,d\right)$ &
$(0, b, 0, b), \quad 1<\lambda$ & $b$ & $(-b, 0, -b, 0)$ & ------\\ \hline
$\dg'_{4,0}$ & $J_1$ & Any & $(0, 0, 0, 0)$ & 0 & $(0, 0, 0, 0)$ & ------\\ \cline{2-7}
& $J_2$ & Any & $(0, 0, 0, 0)$ & 0 & $(0, 0, 0, 0)$ & ------\\ \hline
$\dg'_{4,\delta}$ & $J_1$& $(1,1,1,1)$ & $\frac{3\delta}{2}(1, 1, -1, -1)$
& ------ & ------ & ------\\ \cline{2-7}
& $J_2$ & $(1,1,1,1)$ & $\frac{3\delta}{2}(1, 1, -1, -1)$
& ------ & ------ & ------ \\ \cline{2-7}
& $J_3$ & $(1,1,1,1)$ & $-\frac{3\delta}{2}(1, 1, 1, 1)$
& $-\frac{3\delta}{2}$ & $(0, 0, 0, 0)$ & Yes\\ \cline{2-7}
& $J_4$ & $(1,1,1,1)$ & $-\frac{3\delta}{2}(1, 1, 1, 1)$
& $-\frac{3\delta}{2}$ & $(0, 0, 0, 0)$ & Yes\\ \hline
$\hg_{4}$ & $J$ & $\left(5,2,\frac{5}{4},2\right)$
&$-3(0, 1, 0, 1)$ &
------ & ------ & ------\\ \hline
\end{tabular}
\vspace{0.3cm}
\caption{Chern-Ricci solitons.}\label{tabsoliton}
\end{table}}

\begin{example}\label{muA4}
Assume that $\ggo$ has a codimension-one abelian ideal $\ngo$ (i.e. $\ggo$ is any of the Lie algebras denoted with $\rg$ in Table \ref{tablalgsol} except $\rg_2\rg_2$ and $\rg'_2$).  By Lemma \ref{muAmet}, we can assume that in terms of an orthonormal basis $\{ e_i\}$,
$$
J=\left[\begin{smallmatrix} &&0&-1\\ &&-1&0\\ 0&1&&\\ 1&0&&
\end{smallmatrix}\right], \qquad \omega=e^1\wedge e^4+e^2\wedge e^3,
$$
and the Lie bracket of $\ggo$, denoted by $\mu=\mu_{a,b,c,d,e}$, is given by
$$
\ad_\mu{e_4}|_{\ngo}=\left[\begin{matrix} c&0&0\\ d&a&-b\\ e&b&a
\end{matrix}\right], \qquad c\geq 0.
$$
The Chern-Ricci form and operator are therefore given by
$$
p=-c(c+a) e^1\wedge e^4, \qquad P=-c(c+a)\left[\begin{smallmatrix} 1&&&\\ &0&&\\ &&0&\\ &&&1
\end{smallmatrix}\right].
$$
In the case $p\ne 0$, $\mu$ is a CR-soliton if and only if $d=e=0$ (see Proposition \ref{muA}, (i)).  These are precisely the long time pointed limits one obtains by rescaling CRF-solutions (see Proposition \ref{muA}, (iii)).  By giving different values to $a,b,c$, we have found a CR-soliton for any complex structure on any Lie algebra in this class (see Table \ref{tabsoliton}), with the only exception of $\rg_{4,1}$ (see example below).
\end{example}

\begin{example}\label{r41}
We have that $\mu$ is isomorphic to $\rg_{4,1}$ if and only if $a=c\ne 0$, $b=0$ and at least one of $d,e$ is nonzero, from which it follows that $(\rg_{4,1},J)$ does not admit any CR-soliton metric.  It follows that $\nu_+=\mu_{a,0,a,0,0}\simeq\rg_{4,1,1}$, and so the rescaled solution $\omega(t)/(2t+1)$ converges in the pointed sense, as $t\to\infty$, to the $4$-dimensional real hyperbolic space $\RR H^4$.
\end{example}

\begin{example}\label{exgama}
For any $\gamma\in\RR$, $\delta>0$, consider the solvable Lie algebra $\rg'_{4,\gamma,\delta}$ with Lie bracket as defined in Table \ref{tablalgsol}, which coincides with $\mu_{\gamma,-\delta,1,0,0}$ from Example \ref{muA4}:
$$
\ad_\mu{e_4}|_{\ngo}=\left[\begin{matrix} 1&0&0\\ 0&\gamma&\delta\\ 0&-\delta&\gamma
\end{matrix}\right], \qquad\gamma\in\RR, \qquad\delta>0.
$$
The canonical metric is therefore a CR-soliton for both complex structures $J_1$ and $J_2$, with $p=-(1+\gamma)e^1\wedge e^4$, which is therefore expanding, steady and shrinking for $\gamma>-1$, $\gamma=-1$ and $\gamma<-1$, respectively.  Moreover, $(J,\omega)$ is a K\"ahler-Ricci soliton if and only if $\gamma=0$ (expanding) and for $\gamma=-\unm$, the corresponding Lie group admits a lattice giving rise to a hermitian metric on an Inoue surface of type $S^0$ which is an expanding CR-soliton when pulled back on its universal cover (see \cite{Hsg}).
\end{example}

We have found a compatible CR-soliton for each complex structure on a $4$-dimensional solvable Lie group, with the exceptions of the following seven cases:
$$
(\rg'_2,J_1), \quad (\rg_{4,1},J), \quad (\dg_4,J_2), \quad (\dg_{4,\unm},J_2), \quad  (\dg'_{4,\delta},J_1), \quad (\dg'_{4,\delta},J_2), \quad (\hg_4,J).
$$
We were able to prove the non-existence of a CR-soliton only in the case of $(\rg_{4,1},J)$ (see Example \ref{r41}).  The CR-soliton metrics $g=\omega(\cdot,J\cdot)$ and their respective Chern-Ricci operators $P$ are given in Table \ref{tabsoliton} as diagonal matrices with respect to the basis $\{ e_1,e_2,e_3,e_4\}$, together with the constant $c$ and the derivation $D$ such that $P=cI+D$.  For example, the metric for the complex Lie algebra $(\dg_{4,\unm},\lb_3,J_3)$ is given by $g(e_i,e_j)=\delta_{ij}$ for all $i\ne j$ and
$$
g(e_1,e_1)=2,  \quad g(e_2,e_2)=5,  \quad g(e_3,e_3)=5/4,  \quad g(e_4,e_4)=2.
$$
In the last column we added the condition under which the metric is K\"ahler, that is, a K\"ahler-Ricci soliton.

In the case of $\dg_{4,\lambda}$, in order to simplify the description of the metrics in Table \ref{tabsoliton}, we have introduced the following notation:
$$
a:= -\lambda(\lambda+1), \qquad b:=(1-\lambda)(\lambda-2), \qquad d:=(\lambda-1)^2+1, \qquad e:=\lambda^2+1.
$$

\begin{remark}
There is an infinite family plus five individual solvable Lie groups of dimension $4$ admitting a left-invariant complex structure which also admit a lattice, giving rise to the compact complex surfaces which are solvmanifolds (see \cite{Hsg}). Their Lie algebras are:

\begin{itemize}
\item $\RR^4$: Complex tori.

\item $\rg\hg_3$: Primary Kodaira surfaces.

\item $\rg\rg'_{3,0}$: Hyperelliptic surfaces.

\item $\rg'_{4,-\unm,\delta}$: Inoue surfaces of type $S^0$.

\item $\dg_{4}$: Inoue surfaces of type $S^\pm$.

\item $\dg'_{4,0}$: Secondary Kodaira surfaces.
\end{itemize}
\end{remark}


\end{document}